\documentclass[11pt]{article}
\usepackage{amsmath,amssymb,amsthm,nicefrac,mathrsfs}
\usepackage[colorlinks]{hyperref}  
\newtheorem{theorem}{\sc Theorem}[section]
\newtheorem{lemma}[theorem]{\sc Lemma}
\newtheorem{proposition}[theorem]{\sc Proposition}

\newtheorem{remark}[theorem]{\sc Remark}

\numberwithin{equation}{section}

\newcommand{\be}{\begin{equation}}
\newcommand{\ee}{\end{equation}}

\def\E{\bE}

\def\P{\bP} 

\def\cG{\mathcal{G}}

\def\bE{\mathbb{E}}

\def\bZ{\mathbf{Z}}

\newcommand{\Z}{\mathbf{Z}}

\newcommand{\R}{\mathbf{R}}

\renewcommand{\d}{{\rm d}}

\renewcommand{\geq}{\geqslant}
\renewcommand{\leq}{\leqslant}

\renewcommand{\le}{\leqslant}

\renewcommand{\P}{\mathrm{P}}

\def\m1{\mathbf{1}}

   \DeclareMathOperator{\Cov}{Cov}  

\allowdisplaybreaks[1]

\allowdisplaybreaks[1]

\author{Mohammud Foondun\\ University of Strathclyde \and Eulalia Nualart \\ Universitat Pompeu Fabra\\}

\title{Spatial asymptotics and strong comparison principle for some fractional stochastic heat equations.
\date{}
}

\begin{document}

\maketitle

\begin{abstract}
Consider the following stochastic heat equation,
\begin{align*}
\frac{\partial u_t(x)}{\partial t}=-\nu(-\Delta)^{\alpha/2} u_t(x)+\sigma(u_t(x))\dot{F}(t,\,x), \quad t>0, \; x \in \R^d.
\end{align*}
Here $-\nu(-\Delta)^{\alpha/2}$ is the fractional Laplacian with $\nu>0$ and $\alpha \in (0,2]$, $\sigma: \R\rightarrow \R$ is a globally Lipschitz function, and $\dot{F}(t,\,x)$ is a Gaussian noise which is white in time and colored in space. Under some suitable additional conditions, we explore the effect of the initial data on the spatial asymptotic properties of the solution.   We  also  prove a strong comparison theorem. This constitutes an important extension over a series of works most notably 
\cite{CJK}, \cite{CJKS2}, \cite{CK} and \cite{CDKh}. 

\end{abstract}

\section{Introduction and main results}
Consider the following stochastic heat equation, 
\begin{align}\label{main-eq}
\frac{\partial u_t(x)}{\partial t}=-\nu(-\Delta)^{\alpha/2} u_t(x)+\sigma(u_t(x))\dot{F}(t,\,x),  \quad t>0, \; x \in \R^d,
\end{align}
where $-\nu(-\Delta)^{\alpha/2}$ is the fractional Laplacian, that is, the infinitessimal generator of a symmetric $\alpha$-stable process with density $p_t(x)$, where $\alpha \in (0,2]$, and $\nu>0$ is a viscosity constant. The noise $\dot{F}(t,\,x)$ is white in time and colored in space satisfying 
\begin{equation*}
\Cov(\dot{F}(t,\,x), \dot{F}(s,\,y))=\delta_0(t-s)f(x-y), 
\end{equation*}
where $f$ is the spatial correlation function which we take to be the Riesz kernel
\begin{align*}
f(x):=\frac{1}{|x|^\beta},\qquad 0<\beta<d. 
\end{align*}
The function $\sigma:\R\rightarrow \R$ is a globally Lipschitz continuous function with $\sigma(0)=0$, that is, there exists a constant $L_{\sigma}>0$ such that $$\vert \sigma(x)\vert \leq L_{\sigma} \vert x\vert,\quad \text{  for all } x \in \R^d.$$
The initial condition $u_0$ is always going to be a nonnegative function in $\R^d$ such that $$\bar{u}_0:=\sup_{x\in \R^d}u_0(x)<\infty.$$
 Following Walsh \cite{Walsh}, if one further assume that 
\begin{align*}
\beta<\min(\alpha,\,d),
\end{align*}
then \eqref{main-eq} has a unique mild solution $\{u_t(x), t \geq 0, x \in \R^d\}$ which
is adapted and jointly measurable and satisfies 
\begin{equation}\label{mild}
u_t(x)=(p_t\ast u_0)(x)+\int_0^t\int_{\R^d}p_{t-s}(x-y)\sigma(u_s(y))\,F(\d s\,\d y),
\end{equation}
where 
$$
(p_t\ast u_0)(x)=\int_{\R^d} p_t(x-y) u_0(y) dy,
$$
and 
\begin{equation*}
\sup_{x\in \R^d, t\in [0,\,T]}\E|u_t(x)|^k<\infty\quad\text{for all}\quad k\geq 2\quad\text{and}\quad T<\infty.
\end{equation*}
 For more information about existence-uniqueness considerations, please consult \cite{Walsh}, \cite{minicourse} and \cite{Davar-CBMS}.  
 
 The results of this paper are motivated by two comparison theorems proved recently  in  \cite{FLJ} for the  solution to \eqref{mild}. 
 The first one is the following moment comparison theorem.
\begin{theorem} \label{mct}\textnormal{\cite{FLJ}}
  Let $u$ and $v$ two solutions to \eqref{mild}, one with $\sigma$, the other with another globally Lipschitz continuous function $\bar{\sigma}$ such that $\bar{\sigma}(0)=\sigma(0)=0$ and 
$\sigma(x) \geq \bar{\sigma}(x)\geq 0$ for all $x\in \R_+$. Then for any $k \in \mathbb{N}$, $x \in \R^d$, and $t \geq 0$,
$$
\E[u_t(x)^k] \geq \E[v_t(x)^k].
$$
\end{theorem}
An important consequence of Theorem \ref{mct} are the following sharp estimates on the moments of the solution to (\ref{mild}), when the initial condition is bounded below and under the additional assumption that there exists a constant  $l_\sigma>0$ such that 
\begin{align} \label{sig}
\quad  \sigma(x)\geq l_\sigma|x|,\quad \text{for all}\quad x\in \R^d.
\end{align} 
This was unknown till the work of \cite{FLJ}.
\begin{theorem} \label{momentbounds}
Let $u$ be the  solution to \textnormal{(\ref{mild})}. Assume \textnormal{(\ref{sig})} and 
\begin{equation} \label{infsup}
0<\underline{u}_0:=\inf_{x\in \R^d}u_0(x).
\end{equation}
Then there exists a positive constant $A$ such that for all $x\in \R^d$ and $t>0$,  
\begin{equation*} 
 \frac{\underline{u}_0^k}{A^{k}} \exp\left(\frac{1}{A}k^{\frac{2\alpha-\beta}{\alpha-\beta}}t \nu^{-\frac{\beta}{\alpha-\beta}}\right)\le \E|u_t(x)|^k\le A^k\bar{u}_0^k\exp\left(Ak^{\frac{2\alpha-\beta}{\alpha-\beta}}t\nu^{-\frac{\beta}{\alpha-\beta}}\right).
\end{equation*}
The upper bound holds for $k \geq 2$ while the lower bound holds for $k \geq k_0$, where $k_0$ is some large positive number. 
\end{theorem}
For the  case $\sigma(x)=x$ (known as the Parabolic Anderson model), the above is given by \cite[Lemma 4.1]{KK}. The scaling property of the heat kernel gives the dependence of the bounds on the parameter $\nu$. An immediate consequence of Theorem \ref{momentbounds} is that the solution to (\ref{main-eq}) is fully intermittent meaning that
for all $k \geq 2$,  the function
$$
k  \rightarrow \frac{1}{k} \gamma(k):=\frac{1}{k}\limsup_{t \rightarrow \infty} \log \E|u_t(x)|^k
\quad \text{ is strictly increasing}.
$$
Intuitively, this means that the solution develops many
high peaks distributed over small $x$-intervals when $t$ is large (see \cite{FK} and the references therein). The fact that the solution to (\ref{main-eq}) is weakly intermittent was already known (see e.g.\cite{FLO} and the references therein), meaning that
$$
\gamma(2)>0 \quad \text{ and } \quad \gamma(k)<\infty, \text{ for all } k \geq 2.
$$

The previous results concern the moments of the solution to (\ref{main-eq}), but much less is known about the almost sure asymptotic behaviour of the solution, which is crucial to understand better its chaotic behaviour.
The main purpose of this paper is to explore how the almost surely spatial asymptotic behaviour of the solution to (\ref{main-eq}) depends on the initial function $u_0$. We start with the case that $u_0$ is bounded below as in Theorem \ref{momentbounds}.  A first observation is that, since $u_0$ is also bounded above, then we can easily see that 
\begin{equation*}
\E u_t(x)\leq c,
\end{equation*}
where $c$ is the upper bound of $u_0$. Since $u_0$ is bounded below, it is not trivial to say more about this. However, this is sufficient to show that almost surely, $\liminf_{|x|\rightarrow \infty}u_t(x)$ is bounded as well. This is in sharp contrast with the behaviour of supremum of the solution as described by the next theorem.

\begin{theorem}\label{asym}
Let $u$ be the unique solution to \textnormal{(\ref{mild})}, and assume that \textnormal{(\ref{sig})} and \textnormal{(\ref{infsup})} hold.
Then there exist positive constants $c_1,c_2$ such that for every $t > 0$, 
\begin{align*}
c_1\frac{t^{(\alpha-\beta)/(2\alpha-\beta)}}{\nu^{\beta/(2\alpha-\beta)}} \leq \liminf_{R \rightarrow \infty}&\frac{\log \sup_{x\in B(0,\,R)}u_t(x)}{\left(\log R\right)^{\alpha/(2\alpha-\beta)}}\\
&\leq 
\lim\sup_{R \rightarrow \infty}\frac{\log \sup_{x\in B(0,\,R)}u_t(x)}{\left(\log R\right)^{\alpha/(2\alpha-\beta)}}\leq 
c_2 \frac{t^{(\alpha-\beta)/(2\alpha-\beta)}}{\nu^{\beta/(2\alpha-\beta)}}\quad\text{a.s.}
\end{align*}
\end{theorem}

This theorem is a major improvement of \cite[Theorem 1.3]{CJK} (space-time white noise case) and \cite[Theorem 2.6]{CJKS2} (Riesz kernel spatial covariance). See also \cite{XC} for exact spatial asymptotics when then noise is fractional in  time and correlated in space. All these papers  deal with the  Parabolic Anderson model and  the usual Laplacian ($\alpha=2$). Moreover,  in  \cite{CJK, CJKS2}  the  dependence in  time of the  bounds is not explicit. The case $\sigma(x)=x$, fractional Laplacian and  Riesz kernel spatial covariance is considered in the  preprint \cite[Theorem 1.2]{KK}, without the  dependence on $\nu$ and constant initial data. Obtaining the exact dependence on the viscosity constant $\nu$ is important to understand in which universality class the equation can be associated (see \cite[Remark 1.5]{CJK}).
A key ingredient of the proof of Theorem \ref{asym} are the moment bounds of Theorem \ref{momentbounds}, that will allow to obtain some tail estimates for the solution.

Let us now consider an example where $u_0$ is not bounded below.
\begin{remark}
If $u_0(x):=1_{B(0,1)}(x)$, then one can show that for $x\in B(0,\,R)^c$ and $R$ large enough, we have
\begin{equation*}
\E u_t(x)=(p_t\ast u_0)(x)\leq \frac{ct}{R^{\alpha}}.
\end{equation*}
If we further assume that $\alpha>1$, then a Borel-Cantelli argument shows that $$\liminf_{|x|\rightarrow \infty}u_t(x)=0.$$This indicates that having initial conditions which are not bounded below can influence the behaviour of the solution drastically.
\end{remark}

The above remark can be seen as a motivation for us to drop the assumption that the initial function is bounded below. We have the following trichotomy result, that studies the amount of decay that the initial conditions needs to ensure that the solution is a bounded function a.s. For this result, we restrict ourselves to the case $\alpha=2$, so that the operator is the usual Laplacian instead of the fractional Laplacian.
\begin{theorem}\label{trichotomy}
Let $u$ be the unique solution to \textnormal{(\ref{mild})} with $\alpha=2$.
Assume \textnormal{(\ref{sig})} and that $u_0(x)$ is a radial function satisfying 
\begin{equation*}
\lim_{x\rightarrow \infty}u_0(x)=0\quad \text{and}\quad u_0(x)\leq u_0(y)\quad\text{whenever}\quad x\geq y.
\end{equation*}
Set 
\begin{equation*}
\Lambda:=\lim_{|x|\rightarrow \infty}\frac{|\log u_0(x)|}{(\log |x|)^{2/(2-\beta)}}.
\end{equation*}
Then, if $0<\Lambda<\infty$, there exists a random variable $T$ such that
\begin{align*}
\P\left(\sup_{x\in \R^d} u_t(x)<\infty, \quad \forall t<T\quad \text{and}\quad \sup_{x\in \R^d} u_t(x)=\infty, \quad \forall t>T\right)=1.
\end{align*}
Moreover, if $\Lambda=\infty$, then
\begin{align*}
\P\left(\sup_{x\in \R^d} u_t(x)<\infty, \quad \forall t>0\right)=1.
\end{align*}
Finally, if $\Lambda=0$, then
\begin{align*}
\P\left(\sup_{x\in \R^d} u_t(x)=\infty, \quad \forall t>0\right)=1.
\end{align*}
\end{theorem}
This result is an extension of \cite[Theorem 1.1]{CDKh}, where the case $\alpha=2$ and space-time white noise is  considered. The  proof of their result is  based on the technical Lemma 
\cite[Lemma 2.3]{CDKh} which follows  the ideas of \cite{Chen-Dalang}. Here, we 
use the extension to the spatially colored noise case developed in \cite{ChenKim}. The extension of those techniques to the fractional Laplacian are not straightforward and thus remain open for future work.

Observe that when $u_0$ has compact support corresponds to  the case where $\Lambda=\infty$, and Theorem \ref{trichotomy} shows that the  solution  is bounded for all times a.s.

The second part of  this paper  is motivated by the following  weak comparison principle.
\begin{theorem} \label{weakc}\textnormal{\cite{FLJ}}
Suppose that $u$ and $v$ are two solutions to \eqref{mild} with initial conditions $u_0$ and $v_0$ respectively such that $u_0\leq v_0$. Then
\begin{equation*}
\P(u_t(x)\leq v_t(x)\quad \text{for all}\quad x\in \R^d,\quad t\geq 0)=1.
\end{equation*}
\end{theorem}

Theorem \ref{weakc} ensures nonnegativity of the solution, since the initial condition is assumed to be nonnegative.
For the Parabolic Anderson model, this fact can be deduced from the Feynman-Kac representation of the solution. However, for the general non-linear case, this property for the solution to (\ref{mild}) was unknown until the work of 
\cite{FLJ}.

In this paper we use Theorem \ref{weakc} in order to  show the following strong comparison principle. 
\begin{theorem}\label{strongcomparison}
Suppose that $u$ and $v$ are two solutions to \eqref{mild} with initial conditions $u_0$ and $v_0$ respectively such that $u_0< v_0$. Assume $\alpha \geq 1$. Then
\begin{equation*}
\P(u_t(x)< v_t(x)\quad \text{for all}\quad x\in \R^d,\quad t\geq 0)=1.
\end{equation*}
\end{theorem}

The (strong) comparison principle for equation (\ref{main-eq})  with  space-time white noise and $\alpha=2$ is  the  well-known Mueller's comprison principle (see \cite{Mueller}). Recently, several extensions have been developed. In \cite{CK} the authors extend Mueller's result when the initial  data is more general and there is a more general fractional differential operator than the fractional Laplacian. In  \cite{ChenHuang} the authors  consider  the  non-linear heat equation  in $\R^d$ with a general spatial covariance and measured-valued initial data. The proof  of our  strong comparison principle uses the same strategy as in the papers mentioned above. But the presence of the fractional Laplacian and the colored noise makes it that we have to work a bit harder to prove our result. Moreover, the method of the proof does not seem to extend to the case $\alpha<1$; see Remark \ref{counterexample} for more details.  Among other things, we provide a simplification of the proofs of \cite{ChenHuang} and \cite{CK}. For the  sake of conciseness, we only consider the Riesz kernel spatial covariance. The  extension to general spatial covariances as in \cite{ChenHuang} is left as further work.

As another consequence of the weak comparison  principle (Theorem \ref{weakc}), we show the  next quantitative  result on  the  strict positivity of the solution, which is an extension of \cite[Theorem 5.1]{CJKCorr} (space-time white noise and $\alpha=2$). See also \cite[Theorem 1.4]{CK} and \cite[Theorem 1.6]{ChenHuang}. Not that $\alpha$ is not required to be bigger than $1$.
\begin{theorem}\label{positivity}
Let $T>0$ and $K\subset \R^d$ be a compact set contained in the support of the initial condition $u_0$. Then, there exist constants $c_1$ and $c_2$ depending on $T$ and $K$ such that for all $\epsilon>0$, we have

$$
\P\left( \inf_{t\in [0,\,T]}\inf_{x\in K}u_t(x) < \epsilon \right) \leq c_2 \exp \left( -c_1 \{\vert \log \epsilon \vert  \log \vert \log \epsilon \vert\}^{\frac{2\alpha-\beta}{\alpha}}\right).
$$
\end{theorem}

\vskip 12pt
We now give a plan of the article. In Section 2 we give some preliminary results needed  throughout the paper.
Section 3 is devoted to an approximation result needed for the proof of Theorems \textnormal{\ref{asym}} and \textnormal{\ref{trichotomy}}. These theorems are proved in Section 4. Finally, Section  5 gives the proof of Theorems \textnormal{\ref{strongcomparison}} and \textnormal{\ref{positivity}}.

\section{Preliminary results}
Let $X_t$ be the symmetric $\alpha$-stable process associated with the fractional Laplacian $-\nu (-\Delta)^{\alpha/2}$ and let $p_t(x)$ denote its heat kernel. We will frequently use the following properties.

\begin{itemize} 
\item  {\it Scaling property}: For any positive constant $a$, we have 
\begin{align*}
p_t(x)=a^dp_{a^\alpha t}(ax), \quad \text{ for all } x \in \R^d, t >0. 
\end{align*}
This property follows from 
\begin{equation*}
p_t(x)=(2\pi)^{-d}\int_{\R^d}e^{-ix\cdot z}e^{-t\nu|z|^\alpha}\,\d z.
\end{equation*}
\item {\it Heat kernel estimates} (see \cite{kolo} and references therein): For $0<\alpha<2$, there exist positive constants $c_1$ and $c_2$ such that  for all $x \in \R^d$ and $t>0$,
\begin{align*}
c_1\left(\frac{1}{t^{d/\alpha}}\wedge\frac{t}{|x|^{d+\alpha}}\right)\leq p_t(x)\leq c_2\left(\frac{1}{t^{d/\alpha}}\wedge\frac{t}{|x|^{d+\alpha}}\right).
\end{align*}
\end{itemize}

\begin{remark}
The  proofs of Theorems \textnormal{\ref{asym}} and \textnormal{\ref{trichotomy}} will only  use the  upper bound
\begin{equation} \label{ee2}
p_t(x) \leq c_2\frac{t}{\vert x\vert^{d+\alpha}},\quad  \text{ for sufficiently large } \vert x \vert,
\end{equation}
while the proof of Theorems \textnormal{\ref{strongcomparison}} and \textnormal{\ref{positivity}} will only use the lower bound
\begin{equation} \label{ee1}
p_t(x) \geq c_1 \frac{1}{t^{d/\alpha}},\quad  \text{ for sufficiently small } \vert x \vert.
\end{equation}
Both are also valid for $\alpha=2$.
\end{remark}

The next result provides some estimates  that involve the above heat kernel and the correlation function $f$. They will be useful for proving Lemma \ref{approx}.

\begin{lemma}\label{heat-corr}
There exist positive constants $c_1,c_2$ and $c_3$ such that for all $t>0$, $x\in \R^d$, and $R>0$, we have
\begin{align} \label{a1}
\int_{B(x,\,R)^c\times B(x,\,R)^c}p_{t}(x-y)p_{t}(x-w)f(y-w)\,\d y\,\d w &\leq c_1 \frac{t^2}{R^{2\alpha+\beta}},\\ \label{a2}
\int_{B(x,\,R)^c\times B(x,\,R)}p_{t}(x-y)p_{t}(x-w)f(y-w)\,\d y\,\d w&\leq c_2 \frac{t^{1-\beta/\alpha}}{R^\alpha},\\ \label{a3}
\int_{\R^d\times \R^d}p_{t}(x-y)p_{t}(x-w)f(y-w)\,\d y\,\d w&\leq c_3 t^{-\beta/\alpha}.
\end{align}
\end{lemma}
\begin{proof}
We start with (\ref{a1}).
\begin{align*}
\int_{B(x,\,R)^c\times B(x,\,R)^c}&p_{t}(x-y)p_{t}(x-w)f(y-w)\,\d y\,\d w\\
&\leq \int_{B(0,\,R)^c\times B(0,\,R)^c}p_t(y)p_t(w)f(y-w)\,\d y\,\d w.
\end{align*}
From (\ref{ee1}), the above quantity is bounded by a constant times
\begin{align*}
\frac{t^2}{R^{2\alpha+\beta}}\int_{B(0,\,1)^c\times B(0,\,1)^c}\frac{1}{\vert y\vert^{d+\alpha}\vert w\vert^{d+\alpha}}\frac{1}{\vert
y-w\vert^\beta}\,\d w\,\d y.
\end{align*}
The above integral is finite so the proof of (\ref{a1}) is complete. For (\ref{a2}), we write
\begin{align*}
\int_{B(x,\,R)^c\times B(x,\,R)}&p_{t}(x-y)p_{t}(x-w)f(y-w)\,\d y\,\d w\\
&\leq \int_{B(0,\,R)^c\times \R^d}p_t(y)p_t(w)f(y-w)\,\d y\,\d w.
\end{align*}
By the scaling property,
\begin{align*}
\int_{\R^d}p_t(w)f(y-w)\,\d w=\E^y|X_t|^{-\beta}\leq t^{-\beta/\alpha}\E^0|X_1|^{-\beta}.
\end{align*}
Finally, proceeding  as  before, we get
$$
\int_{B(0,\,R)^c}p_t(y) dy \leq c \frac{t}{R^\alpha}.
$$
Combining the above estimates, we obtain (\ref{a2}). For (\ref{a3}), it suffices to use the  semigroup property 
\begin{align*}
\int_{\R^d\times \R^d}p_{t}(x-y)p_{t}(x-w)f(y-w)\,\d y\,\d w=\int_{\R^d}p_{2t}(w)f(w)\,\d w,
\end{align*}
and using the scaling property as before  we obtain the  desired bound.
\end{proof}

We  now return to $u_t(x)$ the  solution to \textnormal{(\ref{mild})}.
Our next property can be read from \cite{BEM2010a}. For any $k \geq 2$, there exists a positive constant $c:=c(k)$ such that for all $s,t>0$, and $x,\,y\in\R^d $ 
\begin{align*} 
\E|u_s(x)-u_t(y)|^k\leq c\left(\vert x-y\vert^{\eta k}+|s-t|^{\tilde{\eta}k}\right),
\end{align*}
where $\eta=\frac{\alpha-\beta}{2}$ and $\tilde{\eta}=\frac{\alpha-\beta}{2\alpha}$.
The above together with the upper moment bound of Theorem \ref{momentbounds} has the following consequence.
\begin{proposition}\label{cty}
$u_t(x)$ has a continuous version, that is, for any $k \geq 2$, there exist positive constants $c_1,c_2:=c_2(k)$ such that
\begin{align*}
\E\left[\underset{x,\,y\in K\subset \R^d}{\sup_{x\not=y, s\not= t}}\frac{|u_s(x)-u_t(y)|^k}{\vert x-y\vert^{\eta k}+|s-t|^{\tilde{\eta}k}}\right]\leq c_2 e^{c_1k^{(2\alpha-\beta)/(\alpha-\beta)} \nu^{-\beta/(\alpha-\beta)}t}.
\end{align*}
\end{proposition}

\begin{proof}
The proof is very similar to Theorem 4.3 of \cite{minicourse} and is therefore omitted.
\end{proof}

We also have the following property.

\begin{lemma}\label{markov}
Fix $x\in \R^d$, then the solution $u_t(x)$ satisfies the strong Markov Property.
\end{lemma}
\begin{proof}
We omit the proof since it is very similar to \cite[Lemma 3.3]{Mueller}.
\end{proof}

We end this section by recalling the following fractional Gronwall's inequality. 
\begin{proposition} \label{renew}\textnormal{\cite[Lemma 7.1.1]{Henry}}
Let $\rho>0$ and  suppose that $f(t)$ is a locally integrable function  satisfying
\begin{equation*} 
f(t)\leq c_1+k \int_0^t (t-s)^{\rho-1} f(s) d s\quad \text{for\,all} \quad  t >0,
\end{equation*}
for some positive constants $c_1,k$. Then there exist positive constants $c_2,c_3$ such that
\begin{equation*}
f(t)\leq c_2 e^{c_3  k^{1/\rho} t}\quad \text{for\,all}\quad t>0.
\end{equation*}
\end{proposition}

\section{An approximation result}

Theorems \textnormal{\ref{asym}} and \textnormal{\ref{trichotomy}} are almost sure limit theorems and rely on some Borel-Cantelli type arguments.  To be able to carry out the proof, we will need to find an appropriate independent sequence of random variables and it is apriori not clear how to find such a sequence.   We follow \cite{CJK} and \cite{CJKS2} where this issue was successfully resolved. 

Let $n\geq 1$ and consider the following approximation
$F^{(n)}$ of the measure $F$ appearing in (\ref{mild}).  Recall that the covariance of $\dot{F}$ is given by $f(x)=\frac{1}{|x|^\beta}$, $x \in \R^d$. This can be written as $f=h\ast\tilde h$, where $h(x)=\frac{1}{|x|^{\frac{d+\beta}{2}}}$ and $\tilde{h}(x):=h(-x)$. Define $h_n(x):=h(x)Q_n(x)$ and $f_n(x)=(h-h_n)\ast(\tilde{h}-\tilde{h}_n)$, where 
$$
Q_n(x)=\prod_{j=1}^d \left(1-\frac{\vert x_j \vert}{n}\right)_+.
$$
We take $\dot F^{(n)}$ to be the noise satisfying 
\begin{equation*}
\Cov(\dot{F}^{(n)}(t,\,x), \dot{F}^{(n)}(s,\,y))=\delta_0(t-s)g_n(x-y), 
\end{equation*}
where $g_n=h_n \ast \tilde{h}_n$. 

By an argument similar to that of the proof of \cite[Lemma 9.3]{CJKS2}, we have that for any $\gamma\in(0,\,\beta \wedge 1)$ there exists a 
positive constant $c$ such that for all $s>0$ and $n \geq 1$,
\begin{equation*}
(p_s\ast f_n)(0)\leq c \frac{1}{n^{\gamma}}\frac{1}{s^{(\beta-\gamma)/\alpha}}.
\end{equation*}
As a consequence, using the semigroup property, we obtain that
\begin{equation}\label{rty} \begin{split}
 &\int_0^t\int_{\R^d\times \R^d}p_{t-s}(x-y)p_{t-s}(x-z)
f_n(y-z)\, \d y \, \d z \, \d s \\ 
&=\int_0^t\int_{\R^d} p_{2(t-s)}(z)
f_n(z)\,  \d z \,  \d s =\int_0^t (p_{2r} \ast f_n)(0) \d r \\
&\leq c t^{1-\frac{\beta-\gamma}{\alpha}} \frac{1}{n^{\gamma}}.
\end{split}
\end{equation}

Next, consider the following integral equation,
\begin{equation} \label{utn}
U_t^{(n)}(x)=(p_t\ast u_0)(x)+\int_0^t\int_{B(x,\,(nt)^{1/\alpha})}p_{t-s}(x-y)\sigma(U_s^{(n)}(y))F^{(n)}(\d s\,\d y).
\end{equation}
The unique solution to this integral equation can be found via a standard fixed point argument. Fix $n\geq 1$. Set $U_t^{(n, 0)}:=u_0$ and for each $j\geq 1$,  the $j$th Picard iteration is given by
\begin{align*}
U_t^{(n, j)}(x)=\int_{\R^d}p_t(x-y)u_0(y)\,\d y+\int_0^t\int_{B(x,\,(nt)^{1/\alpha})}p_{t-s}(x-y)\sigma(U_s^{(n, j-1)}(y))F^{(n)}(\d s\,\d y).
\end{align*}
Moreover, one can show that under our current standing conditions, the unique solution satisfies for all $t>0$ and $k \geq 2$,
\begin{equation*}
\sup_{n \geq 1}\sup_{s\in [0,\,t]} \sup_{x \in \R^d}\E|U_s^{(n)}(x)|^k\leq c_2 e^{c_1 k^{(2\alpha-\beta)/(\alpha-\beta)}t},
\end{equation*}
for some positive constants $c_1,c_2(k)$. As a consequence, for all $t>0$, $k \geq 2$, and sufficiently large $n$,
\begin{equation} \label{usup}
\sup_{s\in [0,\,t]} \sup_{x \in \R^d}\E|U_s^{(n,n-1)}(x)|^k\leq c_2 e^{c_1 k^{(2\alpha-\beta)/(\alpha-\beta)}t},
\end{equation}
for some positive constants $c_1,c_2(k)$.  We also have the following result which gives us the independent quantities we need. 
\begin{lemma}\label{ind}
Let $t>0$ and $n\geq 1$. Suppose that $\{x_i \}_{i=1}^\infty\subset \R^d$ with $\vert x_i-x_j\vert\geq 2n^{1+1/\alpha}t^{1/\alpha}$ for all $i\not=j$. Then $\{U^{(n,n)}_t(x_i)\}_{i=1}^\infty$ are independent random variables.
\end{lemma}
\begin{proof}
The proof is similar to that of \cite[Lemma 5.4]{CJKS2} and is omitted.
\end{proof}

We will also need the fact that the random variables defined above approximate the solution to \eqref{main-eq}. We provide a proof of this fact next.
\begin{lemma}\label{approx}
For all $T >0$ and $k \geq 2$, there exist positive constants $c_1$ and $c_2$ such that for large enough $n$, 
\begin{align*}
\sup_{t \in [0,T]} \sup_{x \in \R^d}\E|u_t(x)-U_t^{(n, n)}(x)|^k\leq c_2  \frac{1}{n^{\gamma k/2}}e^{c_1k^{(2\alpha-\beta)/(\alpha-\beta)}t}.
\end{align*}
\end{lemma}
\begin{proof}
Consider the following integral equation,
\begin{equation*} 
V_t^{(n)}(x)=(p_t\ast u_0)(x)+\int_0^t\int_{B(x,\,(nt)^{1/\alpha})}p_{t-s}(x-y)\sigma(V_s^{(n)}(y))F(\d s\,\d y).
\end{equation*}
We first look at $V^{(n)}_t(x)-U_t^{(n, n)}(x)$ and its moments.
\begin{align*}
V^{(n)}_t(x)-U_t^{(n,n)}(x)&=\int_0^t\int_{B(x,\,(nt)^{1/\alpha})}p_{t-s}(x-y)\sigma(V_s^{(n)}(y))F(\d s\,\d y)\\
&-\int_0^t\int_{B(x,\,(nt)^{1/\alpha})}p_{t-s}(x-y)\sigma(U_s^{(n,n-1)}(y))F^{(n)}(\d s\,\d y).
\end{align*}
We rewrite the above as 
\begin{align*}
V^{(n)}_t(x)-U_t^{(n,n)}(x)&=\int_0^t\int_{B(x,\,(nt)^{1/\alpha})}p_{t-s}(x-y)[\sigma(V_s^{(n)}(y))-\sigma(U_s^{(n,n-1)}(y))]F(\d s\,\d y)\\
&-\int_0^t\int_{B(x,\,(nt)^{1/\alpha})}p_{t-s}(x-y)\sigma(U_s^{(n,n-1)}(y))[F^{(n)}(\d s\,\d y)-F(\d s\,\d y)]\\
&:=I_1+I_2.
\end{align*}
We start bounding $I_2$. 
Using Burkholder-Davis-Gundy inequality and Minkowski's inequalities together with (\ref{usup}), we get that
\begin{equation*}
\begin{split}
\E \vert I_2 \vert^k\leq c_2 e^{c_1 k^{(2\alpha-\beta)/\alpha-\beta}t}\left( \int_0^t\int_{\R^d\times \R^d}p_{t-s}(x-y)p_{t-s}(x-z)
f_n(y-z)\, \d y \, \d z \, \d s\right)^{k/2}.
\end{split}
\end{equation*} 
Appealing to (\ref{rty}), we conclude that
\begin{equation*}
\sup_{t \in [0,T]} \sup_{x \in \R^d}\E \vert I_2 \vert^k\leq c_2(T)\frac{e^{c_1 k^{(2\alpha-\beta)/\alpha-\beta}t}}{n^{\gamma k/2}}.
\end{equation*}

We next treat $I_1$. We look at $U^{(n,n)}_t(x)-U_t^{(n, n-1)}(x)$. Using Burkholder-Davis-Gundy and Minkowski's inequalities together with Lemma \ref{heat-corr}(c), we obtain
\begin{equation*}
\begin{split}
\E|U_t^{(n,n)}(x)-U_t^{(n,n-1)}(x)|^k\leq c(k) \sup_{s \in [0,t]}\sup_{x \in \R^d} \E|U_s^{(n,n-1)}(x)-U_s^{(n,n-2)}(x)|^k \left( \int_0^t  s^{-\beta/\alpha}\d s\right)^{k/2}.
\end{split}
\end{equation*} 
Iterating $n$ times this procedure and choosing $T \leq 1/2$, we get 
\begin{equation*} \begin{split}
\sup_{t \in [0,T]} \E|U_t^{(n,n)}(x)-U_t^{(n,n-1)}(x)|^k&\leq c(k) T^{nk/2} \leq c(k) \left(\frac{1}{2}\right)^{nk/2} \\
&\leq c(k) \frac{1}{n^{\gamma k/2}}.
\end{split}
\end{equation*}
Splitting the interval $[0,T]$ into subintervals of length $\frac12$, we deduce that for all $T>0$,
\begin{equation*} \begin{split}
\sup_{t \in [0,T]} \E|U_t^{(n,n)}(x)-U_t^{(n,n-1)}(x)|^k\leq c(k,T) \frac{1}{n^{\gamma k/2}}.
\end{split}
\end{equation*}
We next set
\begin{equation*}
\mathcal{D}^n_t:=\sup_{ x\in \R^d}\E|V^{(n)}_t(x)-U_t^{(n,n)}(x)|^k.
\end{equation*}
Using Burkholder-Davis-Gundy inequality and Minkowski's inequalities, together with Lemma \ref{heat-corr}(c), and adding and substracting the term $U_s^{(n,n)}(y)$, we obtain
\begin{equation*}
\E \vert I_1 \vert^k\leq c(k,T) \int_0^t\frac{\mathcal{D}^n_s+n^{-\gamma k/2}}{(t-s)^{\beta/\alpha}}\,\d s
\end{equation*}
Using Proposition \ref{renew},
\begin{equation*}
\E \vert I_1 \vert^k\leq c_2 \left(\int_0^t\frac{\mathcal{D}^n_s}{(t-s)^{\beta/\alpha}}\,\d s+
\frac{e^{c_1 k^{(2\alpha-\beta)/\alpha-\beta}t}}{n^{\gamma k/2}}\right).
\end{equation*}

Combining the bound for $I_2$ and $I_1$, we obtain 
\begin{align*}
\mathcal{D}^n_t\leq c_2 \left( \frac{e^{c_1 k^{(2\alpha-\beta)/\alpha-\beta}t}}{n^{\gamma k/2}}+\int_0^t\frac{\mathcal{D}^n_s}{(t-s)^{\beta/\alpha}}\,\d s\right).
\end{align*}
By an appropriate use of Proposition \ref{renew}, we conclude that
\begin{equation} \label{d1}
\mathcal{D}^n_t\leq c_2 \frac{e^{c_1k^{(2\alpha-\beta)/\alpha-\beta}t}}{n^{\gamma k/2}}.
\end{equation} 
We now look at $u_t(x)-V_t^{(n)}(x)$ to obtain 
\begin{align*}
u_t(x)-V_t^{(n)}(x)&=\int_0^t\int_{B(x,\,(nt)^{1/\alpha})}p_{t-s}(x-y)[\sigma(u_s(y)-V_s^{(n)}(y))]F(\d s\,\d y)\\
&+\int_0^t\int_{B(x,\,(nt)^{1/\alpha})^c}p_{t-s}(x-y)\sigma(u_s(y))F(\d s\,\d y),
\end{align*}
which gives us
\begin{align*}
\E|u_t(x)-V_s^{(n)}(x)|^k&\leq c\bigg( \E\left|\int_0^t\int_{B(x,\,(nt)^{1/\alpha})}p_{t-s}(x-y)[\sigma(u_s(y)-V_s^{(n)}(y))]F(\d s\,\d y)\right|^k\\
&+\E\left|\int_0^t\int_{B(x,\,(nt)^{1/\alpha})^c}p_{t-s}(x-y)\sigma(u_s(y))F(\d s\,\d y)\right|^k \bigg)\\
&:=I_1+I_2.
\end{align*}
We bound the second term first. Using the bound on the moments of the solution together with Lemma \ref{heat-corr}(a), we obtain
\begin{align*}
I_2&\leq c_2 e^{c_1 k^{(2\alpha-\beta)/(\alpha-\beta)}t}\left[ \int_0^t\int_{B(x,\,(nt)^{1/\alpha})^{2,c}}p_{t-s}(x-y)p_{t-s}(x-w)f(y-w)\,\d y\,\d w\, \d s\right]^{k/2}\\
&\leq c_2 \frac{1}{n^{(2+\beta/\alpha)k/2}}e^{c_1 k^{(2\alpha-\beta)/(\alpha-\beta)}t}.
\end{align*}
We now consider the first term.
\begin{align*}
I_1&\leq c \int_0^t\sup_{y\in \R^d}\E|u_s(y)-V_s^{(n)}(y)|^k\int_{\R^d\times\R^d}p_{t-s}(x-y)p_{t-s}(x-w)f(y-w)\,\d y\,\d w\, \d s\\
&\leq c \int_0^t \sup_{y\in \R^d}\E|u_s(y)-V_s^{(n)}(y)|^k\frac{1}{(t-s)^{\beta/\alpha}}\,\d s.
\end{align*}
Putting these two bounds together and using Proposition \ref{heat-corr}, we obtain 
\begin{equation} \label{d2}
\sup_{s \in [0,T]} \sup_{x \in \R^d}\E|u_s(x)-V_s^{(n)}(x)|^k\leq c_2 \frac{1}{n^{(2+\beta/\alpha)k/2}}e^{c_1 k^{(2\alpha-\beta)/(\alpha-\beta)}t}.
\end{equation}
Combining the estimates (\ref{d1}) and (\ref{d2}) and using the fact that $\gamma<2+\beta/\alpha$ we obtain the required result.
\end{proof}

\section{Proof of the spatial asymptotic results}

In this section we give the proof of Theorems  \textnormal{\ref{asym}} and \textnormal{\ref{trichotomy}}. We start with several preliminary results.

\subsection{Tail estimates I}

This subsection is devoted to the proof of two tail estimates which are a consequence of the sharp moment estimates in Theorem \ref{momentbounds}.
\begin{lemma} \label{uppertail}
There exists a constant $c_{A,\alpha,\beta}>0$ such that for all $\lambda>0$ and $t>0$,
\begin{align*}
\sup _{x \in \R^d} \P(u_t(x)>\lambda)\leq  \exp\left(-\frac{c_{A,\alpha,\beta} \nu^{\beta/\alpha}}{t^{(\alpha-\beta)/\alpha}}\left|\log \frac{\lambda}{A\bar{u}_0}\right|^{(2\alpha-\beta)/\alpha}\right),
\end{align*}
where $A$ and $\bar{u}_0$ are defined in Theorem \textnormal{\ref{momentbounds}}.
\end{lemma}
\begin{proof}
We start by using Chebyshev's inequality to obtain,
\begin{align*}
\P(u_t(x)>\lambda)&\leq \frac{1}{\lambda^k}\E|u_t(x)|^k\\
&\leq A^k\bar{u}_0^k\lambda^{-k}e^{Ak^{(2\alpha-\beta)/(\alpha-\beta)} \nu^{-\beta/(\alpha-\beta)}t}\\
&\leq \exp\left(Ak^{\frac{2\alpha-\beta}{\alpha-\beta}} \nu^{-\beta/(\alpha-\beta)}t-k\log \frac{\lambda}{A\bar{u}_0}\right).
\end{align*}
The function $F(k):=Ak^{\frac{2\alpha-\beta}{\alpha-\beta}} \nu^{-\beta/(\alpha-\beta)}t-k\log \frac{\lambda}{A\bar{u}_0}$ is optimised at the point 
\begin{equation*}
k^\ast=\left[\frac{ \nu^{\beta/(\alpha-\beta)}}{At}\left(\frac{\alpha-\beta}{2\alpha-\beta}\right)\left(\log \frac{\lambda}{A\bar{u}_0}\right) \right]^{(\alpha-\beta)/\alpha}.
\end{equation*}
Some computations then give
\begin{equation*}
\P(u_t(x)>\lambda)\leq  \exp\left(-\frac{c_{A,\alpha,\beta} \nu^{\beta/\alpha}}{t^{(\alpha-\beta)/\alpha}}\left|\log \frac{\lambda}{A\bar{u}_0}\right|^{(2\alpha-\beta)/\alpha}\right),
\end{equation*}
where $c_{A,\alpha,\beta}=\frac{\alpha}{2\alpha-\beta}\left[\frac{1}{A}\left(\frac{\alpha-\beta}{2\alpha-\beta}\right) \right]^{(\alpha-\beta)/\alpha}.$
\end{proof}

\begin{lemma} \label{lowertail}
Fix $t>0$. Set $\lambda:=\frac{\underline{u}_0}{2A} e^{tk^{\alpha/(\alpha-\beta)} \nu^{-\beta/(\alpha-\beta)}/A}$ for $k\geq k_0$, where $k_0$ is a large number. Then there exists a constant $\tilde{c}_{A,\alpha,\beta}>0$, 
\begin{align*}
\inf _{x \in \R^d} \P(u_t(x)>\lambda)\geq \frac14 \exp\left(-\frac{\tilde{c}_{A,\alpha,\beta} \nu^{\beta/\alpha}}{t^{(\alpha-\beta)/\alpha}} \left(\log \frac{2\lambda A}{\underline{u}_0}\right)^{(2\alpha-\beta)/\alpha}\left(1+\frac{1}{\log \frac{2\lambda A}{\underline{u}_0}} \right)\right).
\end{align*}
The quantities $A$ and $\underline{u}_0$ are defined in Theorem \textnormal{\ref{momentbounds}}.
\end{lemma}
\begin{proof}
By Paley-Zygmund inequality, we have for all $k \geq 2$,
\begin{align*}
\P(u_t(x)\geq \frac{1}{2}\|u_t(x)\|_{L^{2k}(\Omega)})\geq \frac{(\E|u_t(x)|^{2k})^2}{4\E|u_t(x)|^{4k}}.
\end{align*}
Set $\lambda:=\frac{\underline{u}_0}{2A}e^{tk^{\alpha/(\alpha-\beta)} \nu^{-\beta/(\alpha-\beta)}/A}$. Taking into account the bounds on the moments, we obtain 
\begin{align} \label{aux1}
\P(u_t(x)\geq \lambda)\geq \frac{1}{4}\exp\left(-c_{A, \alpha,\beta}k^{(2\alpha-\beta)/(\alpha-\beta)} \nu^{-\beta/(\alpha-\beta)}t+k\log \left(\frac{\tilde{u}_0^4}{A^8}\right)\right),
\end{align}
where $c_{A, \alpha,\beta}:=2^{(2\alpha-\beta)/(\alpha-\beta)}[A2^{(2\alpha-\beta)/(\alpha-\beta)}-\frac{2}{A}]$ and $\tilde{u}_0=\frac{\underline{u}_0}{\bar{u}_0}.$  Finally, some computations we get the desired bound.
\end{proof}

\subsection{Insensitivity analysis}

The next theorem is crucial in the proof of the spatial asymptotic result when the initial condition is not bounded below. Intuitively, we study how the solution is sensible to changes to the initial data, and we conclude that when $R$ is large, the values of the solution in a given ball of radius $R$ are insensitive to the changes of the initial value outside the ball. Here $\alpha=2$, so that the operator now is the usual Laplacian instead of the fractional Laplacian.
\begin{theorem}\label{sensitivity}
Let $a\in \R^d$ and $R>1$. Let $u$ and $v$ be the solution to \eqref{mild} for $\alpha=2$ with respective initial conditions $u_0$ and $v_0$.  Suppose that on $B(a,\,2R)$, $u_0(x)=v_0(x)$ and $v_0(x)\geq u_0(x)$ or $v_0(x)\leq u_0(x)$ everywhere else.  Then, there exists a function $g(t)$ such that for all $t>0$,
\begin{equation*}
\sup_{x\in B(a,\,R)}\E|u_t(x)-v_t(x)|^2\leq g(t) \|u_0-v_0\|^2_{L^\infty(\R^d)}e^{-\frac{R^2}{t}}.
\end{equation*}
\end{theorem}

\begin{proof}
The proof will relies on computations done in \cite{ChenKim}. More precisely, we make use of the proof of part(2) of Theorem 2.4 of that paper. We first introduce some notations. Set 
\begin{align*}
E_t(x)&:= (\cG u_0)_t(x)-(\cG v_0)_t(x)\\
D_t(x,\,\tilde{x})&:=\E[u_t(x)-v_t(x)][u_t(\tilde{x})-v_t(\tilde{x})]\\
D^\sigma_t(x,\,\tilde{x})&:=\E[\sigma(u_t(x))-\sigma(v_t(x))][\sigma(u_t(\tilde{x}))-\sigma(v_t(\tilde{x}))].
\end{align*}
From the mild solution and the fact $\sigma$ is globally Lipschitz, we have 
\begin{align*}
&D_t(x,\,\tilde{x})\\
&=E_t(x)E_t(\tilde{x})+\int_0^t\iint_{\R^d\times \R^d}p_{t-s}(x-y_1)p_{t-s}(x-\tilde{y}_1)f(y_1, \tilde{y}_1)D^\sigma_s(y_1,\,\tilde{y}_1)\,\d y_1\,\d \tilde{y}_1 \d s\\
&\leq E_t(x)E_t(\tilde{x})+c_1\int_0^t\iint_{\R^d\times \R^d}p_{t-s}(x-y_1)p_{t-s}(x-\tilde{y}_1)f(y_1, \tilde{y}_1)D_s(y_1,\,\tilde{y}_1)\,\d y_1\,\d \tilde{y}_1 \d s.
\end{align*}
The recursive argument used in the proof of Theorem 2.4 of \cite{ChenKim} together with inequality (2.19) of Lemma 2.7 of the same paper yield the following bound,
\begin{align*}
D_t(x,\,\tilde{x})&\leq F(t)E_t(x)E_t(\tilde{x}),
\end{align*}
where $F(t)$ is some function of $t$. Other than the fact that it is well defined for any $t>0$, we won't need any other information on this function. We now use Lemma 2.2 of \cite{CDKh}, to see that 
\begin{align*}
\sup_{x\in B(a,\,R)}E_t(x)\leq c_2 \|u_0-v_0\|_{L^\infty(\R^d)}e^{-R^2/(2t)}.
\end{align*}
We combine the above estimate with $x=\tilde{x}$ to see that there exists a function $g(t)$ such that 
\begin{equation*}
\sup_{x\in B(a,\,R)}\E|u_t(x)-v_t(x)|^2\leq g(t) \|u_0-v_0\|^2_{L^\infty(\R^d)}e^{-\frac{R^2}{t}}.
\end{equation*}
\end{proof}

\subsection{Tail Estimates II}

In this subsection we are going to prove tail estimates when the initial condition is not bounded below. The next result is an extension of \cite[Theorem 2.4]{CDKh}. Here, we are still in the case $\alpha=2$.
\begin{theorem}\label{tail:asymp}
Suppose that $u$ and $u_0$ are as  in Theorem \textnormal{\ref{trichotomy}}. Then, there exist positive constants $K_1,K_2$ such that for all $\lambda>0$,
\begin{align*}
-K_1\frac{\Lambda^{2-\beta/2}}{t^{1-\beta/2}}\leq \liminf_{|x|\rightarrow \infty}\frac{\log \P(u_t(x)>\lambda)}{\log |x|}\leq \limsup_{|x|\rightarrow \infty}\frac{\log \P(u_t(x)>\lambda)}{\log |x|}\leq -K_2\frac{\Lambda^{2-\beta/2}}{t^{1-\beta/2}},
\end{align*}
uniformly for all $t$ in every fixed compact subset of $(0, \infty)$.
\end{theorem}
\begin{proof}
We prove the lower bound first. Fix $a \in \R^d$. Let $w_t$ be the solution to \eqref{main-eq} when the initial condition is given by the following
\begin{align*}
w_0(x):=u_0(|x|\vee(3|a|))\quad \text{for all}\quad x\in \R^d.
\end{align*}
Since $w_0\leq u_0$, the weak comparison principle Theorem \ref{weakc} tells us that for all $t>0$ and $x\in \R^d$,
\begin{align*}
w_t(x)\leq u_t(x). 
\end{align*}
This means that finding a lower bound on the tail distribution of $u_t(x)$ amounts to finding a lower bound for the corresponding distribution of $w_t(x)$. 

Now, let $u_t^a(x)$ be the solution to \eqref{main-eq} with initial condition $u_0(3 \vert a \vert)$. Fix $\lambda>0$.
Then, by Theorem \ref{sensitivity}, whenever $R=|a|>1$,
\begin{align*}
\sup_{x\in B(a,\,|a|)}\P(|u_t^a(x)-w_t(x)|>\lambda)&\leq\sup_{x\in B(a,\,|a|)}\frac{\E|u_t^a(x)-w_t(x)|^2}{\lambda^2}\\
&\leq g(t) \frac{1}{\lambda^2 }e^{-\vert a \vert^2/t},
\end{align*}
where $g(t)$ is independent of $a$.

Recall that $u_0(3\vert a\vert)$ is decreasing in $a$ and 
\begin{equation*}
\lim_{a\rightarrow \infty}u_0(3\vert a \vert)=0.
\end{equation*}
Let $k_0$ to be as in Lemma \ref{lowertail}, we can take $\vert a \vert$ large enough so that 
$$
k:=\left(\frac{A\nu^{\frac{\beta}{2-\beta}}}{t}\log \left(\frac{4A \lambda}{u_0(3\vert a \vert)}\right)\right)^{1-\beta/2} \geq k_0,
$$
and
\begin{equation*}
\left|\log \left(\frac{4\lambda A}{u_0(3\vert a \vert)}\right)\right| \geq \frac12.\end{equation*}
We now use Lemma \ref{lowertail} to obtain
\begin{align*}
\inf_{x\in B(a,\,|a|)}\P(u^a_t(x)\geq 2\lambda)\geq \frac14 \exp\left(-\frac{c_{A,\beta} \nu^{\beta/2}}{t^{1-\beta/2}} \left(\log \frac{4\lambda A}{\underline{u}_0}\right)^{2-\beta/2}\right).
\end{align*}
Upon taking $\vert a \vert$ larger if required so that 
\begin{equation*}
 \left|\log \frac{4\lambda A}{u_0(3\vert a \vert)}\right|\leq 2|\log u_0(3\vert a \vert)|,
\end{equation*}
we can use the above together with the definition of $\Lambda$ to write
\begin{equation*} \begin{split}
\inf_{x\in B(a,\,|a|)}\P(u_t(x)\geq \lambda)&\geq \inf_{x\in B(a,\,|a|)}\P(u^a_t(x)\geq 2\lambda)-\sup_{x\in B(a,\,|a|)}\P(|u_t^a(x)-w_t(x)|>\lambda)\\
&\geq \frac14  \exp\left(-\frac{\tilde{c}_{A,\beta} \nu^{\beta/2}\Lambda^{2-\beta/2}}{t^{1-\beta/2}}\log |a|\right) - g(t) \frac{1}{\lambda^2 }e^{-\vert a \vert^2/t}.
\end{split}
\end{equation*}
The above immediately gives the lower bound needed. We now turn our attention to the upper bound. The proof uses a similar strategy to the one of the upper bound. We look at $w_t$, the solution to \eqref{main-eq} but this time, the initial condition is defined by 
\begin{align*}
w_0(x):=u_0(|x|\wedge 2|a|),
\end{align*}
so that now we have $w_0(x)\geq u_0(x)$ which gives us $w_t(x)\geq u_t(x)$ by the weak comparison principle. Now consider $u^a_t$ a solution with constant initial condition given by 
\begin{align*}
z_0(x):=u_0(2|a|).
\end{align*}
We choose $a$ large enough such that
\begin{equation*}
\left|\log \frac{\lambda}{2Au_0(2|a|)}\right|\geq \frac{|\log u_0(2|a|)|}{2}.
\end{equation*}
Then, by Theorem \ref{sensitivity} and Lemma \ref{uppertail}, for $\vert a \vert$ large enough
\begin{equation} \label{tail2}\begin{split}
\sup_{x\in B(a,\,|a|)}\P(u_t(x)>\lambda)&\leq\sup_{x\in B(a,\,|a|)} \P(w_t(x)>\lambda)\\
 &\leq \sup_{x\in B(a,\,|a|)}\P(u_t^a(x)>\lambda/2)+\sup_{x\in B(a,\,|a|)}\P(|u_t^a(x)-w_t(x)|>\lambda/2)\nonumber\\
&\leq \exp\left(-\frac{c_{A,\beta}\Lambda^{2-\beta/2}}{
t^{1-\beta/2}}\log |a|\right)+g(t) \frac{1}{\lambda^2}e^{-\vert a \vert^2/t},
\end{split}
\end{equation}
which implies the desired upper bound and thus finishes the proof.
\end{proof}

\subsection{Proof of Theorem \ref{asym}}
\begin{proof}[Proof of Theorem \ref{asym}.]
Let $t>0$ and set
\begin{align*}
L:= \frac{\underline{u}_0}{6A}\exp\left[\delta_1t^{(\alpha-\beta)/(2\alpha-\beta)}\left|\log R\right|^{\alpha/(2\alpha-\beta)} \nu^{-\beta/(2\alpha-\beta)}\right],
\end{align*}
where $\delta_1$ be a positive constant. Then, we choose $$k=(A \delta_1)^{(\alpha-\beta)/\alpha}\left(\frac{\vert \log(R) \vert}{t}\right)^{(\alpha-\beta)/(2\alpha-\beta)}\nu^{\beta/(2\alpha-\beta)}$$
so that $L$ becomes 
\begin{align*}
L:= \frac{\underline{u}_0}{6A}e^{tk^{\alpha/(\alpha-\beta)} \nu^{-\beta/(\alpha-\beta)}/A}.
\end{align*}
We now apply inequality (\ref{aux1}) to obtain for sufficiently large $R$,
\begin{align*}
&\P(u_t(x)>3L)\\
&\geq \frac{1}{4}\exp\bigg(-c_{A, \alpha,\beta}(A\delta_1)^{(2\alpha-\beta)/\alpha} \vert \log R \vert\\
&\qquad \qquad \qquad -\log \left(A^8\right)(A\delta_1)^{(\alpha-\beta)/\alpha} \left(\vert \log R \vert/t\right)^{(\alpha-\beta)/(2\alpha-\beta)}\nu^{\beta/(2\alpha-\beta)}\bigg) \\
&\geq \frac{1}{4 R^{\delta_2}},  
\end{align*}
where $\delta_1$ is chosen such that $\delta_2<2$. 

Let $N>0$ and choose $x_1,x_2,\ldots, x_N \in \R^d$ such that $\vert x_i-x_j\vert \geq 2n^{1+1/\alpha}t^{1/\alpha}$ for $i\not=j$.  Lemma \ref{ind} then implies that the $U^{(n,n)}_t(x_i)$'s are independent for large enough $n$. We have 
\begin{align*}
\P(\max_{1\leq i\leq N}u_t(x_i)<L)&\leq \P(\max_{1\leq i\leq N}|U^{(n,n)}_t(x_i)|<2L)\\
&\qquad +\P(|u_t(x_i)-U_t^{(n,n)}(x_i)|>L\quad\text{for some}\quad1\leq i\leq N)\\
&:=I_1+I_2.
\end{align*}
We will look at the second term first. By Lemma \ref{approx}, for all $k \geq 2$ and large $n$,
\begin{align*}
I_2\leq \frac{N\E|u_t(x_i)-U_t^{(n,n)}(x_i)|^k}{L^k}\leq c_2  \frac{N}{n^{\gamma k/2}}e^{c_1k^{(2\alpha-\beta)/(\alpha-\beta)}t},
\end{align*}
where we have chosen $R$ large enough such that $L \geq 1$.

We now choose $n\geq N^{10/(3\gamma)}$ so that we have 
\begin{align*}
I_2\leq c_2\frac{1}{N^{2/3}}e^{c_1k^{(2\alpha-\beta)/(\alpha-\beta)}t}.
\end{align*}
Upon choosing $N$ to be an integer greater than $R^3$, we obtain 
\begin{align*}
I_2 \leq c(T,k) \frac{1}{R^2}.
\end{align*} 
To bound $I_1$, we have for large enough $R$,
\begin{align*}
\P(U^{(n,n)}_t(x_i)\geq 2L)&\geq \P(|u_t(x_i)|\geq 3L)-\P(|u_t(x_i)-U_t^{(n,n)}(x_i)|\geq L)\\
&\geq c\left(\frac{1}{R^{\delta_2}} -\frac{1}{R^2}\right)\geq  \frac{c}{R^2}.
\end{align*}
By independence, we have 
\begin{align*}
I_1 \leq  \left(1-\P(U^{(n,n)}_t(x_i)\geq 2L)\right)^N.
\end{align*}
Combining the above and bearing in mind that $N$ is larger than $R^3$, we obtain 
\begin{align*}
\P(\max_{1\leq i\leq N}u_t(x_i)<L)&\leq \left(1-\P(U^{(n,n)}_t(x_i)\geq 2L)\right)^N+\frac{c}{R^2}\\
&\leq\frac{c}{R^2},
\end{align*}
for $R$ large enough.
And hence by a standard monotonicity argument, we have
\begin{align*}
\P\left(\sup_{x\in B(0,\,R)}u_t(x)\leq \frac{\underline{u}_0}{6A}\exp\left[\delta_1t^{(\alpha-\beta)/(2\alpha-\beta)}\left(\log R\right)^{\frac{\alpha}{(2\alpha-\beta)}}\nu^{-\beta/(2\alpha-\beta)}\right]\right)\leq\frac{c}{R^2}.
\end{align*}
We now use Borel Cantelli lemma to obtain that almost surely, for $R\rightarrow \infty$, we have 
\begin{align*}
\sup_{x\in B(0,\,R)}u_t(x)\geq \frac{\underline{u}_0}{6A}\exp\left[\delta_1t^{(\alpha-\beta)/(2\alpha-\beta)}\left(\log R\right)^{\frac{\alpha}{2\alpha-\beta}}\right],
\end{align*}
which  concludes  the proof of the  lower bound.
We now prove the upper bound.  Set
\begin{align*}
U:= A\bar{u}_0\exp\left[\delta_3t^{(\alpha-\beta)/(2\alpha-\beta)}\left(\log R\right)^{\frac{\alpha}{2\alpha-\beta}}\nu^{-\beta/(2\alpha-\beta)}\right],
\end{align*}
for some positive constant $\delta_3$. For $x\in \bZ^d$, denote the cube of side length $1$ by $Q_x$. Let $R$ be a positive integer and decompose $[-R,\,R]^d$ into cubes of the form $Q_x$ so that  $[-R,\,R]^d=\cup_{x\in S}Q_x$ where $S$ is some set of finite cardinality. By Proposition \ref{cty}, for any  $x\in \R^d$ and $k \geq 2$, we have
\begin{align}\label{spatial-cty}
\E\left[\sup_{w,\,y\in Q_x}|u_t(w)-u_t(y)|^{k}\right]\leq c_2 \exp\left[c_1 k^{(2\alpha-\beta)/(\alpha-\beta)}t\right].
\end{align}
We can now write 
\begin{align*}
\P\left(\sup_{x\in [-R,\,R]^d}u_t(x)\geq 2 U\right)&\leq \P\left(\max_{x\in S}u_t(x)\geq U\right)+\P\left(\max_{x\in S}\sup_{y\in Q_x}|u_t(y)-u_t(x)|\geq U \right)\\
&:=I_1+I_2.
\end{align*}
To bound $I_1$, we use Lemma \ref{uppertail} to obtain
\begin{align*}
I_1\leq |S|\P(u_t(x)\geq U)\leq \frac{c}{R^{\delta_4}},
\end{align*}
where the constant $\delta_3$ is chosen so that $\delta_4>1$.
We now bound $I_2$ by making use of \eqref{spatial-cty},
\begin{align*}
I_2\leq |S|\P\left(\sup_{y\in Q_x}|u_t(y)-u_t(x)|\geq U \right)\leq \frac{c_2 |S|\exp(c_1 k^{(2\alpha-\beta)/(\alpha-\beta)}t)}{\exp(k\delta_3 t^{(\alpha-\beta)/(2\alpha-\beta)}(\log R)^{\alpha/(2\alpha-\beta)}\nu^{-\beta/(2\alpha-\beta)})}.
\end{align*}
We now set $k=\delta_5\left(\frac{\log R}{t}\right)^{(\alpha-\beta)/(2\alpha-\beta)}\nu^{\beta/(2\alpha-\beta)}$ to obtain 
\begin{align*}
I_2\leq\frac{c}{R^{\delta_6}},
\end{align*}
where we choose $\delta_3$ so that $\delta_6>1$.
 We can conclude that 
\begin{align*}
\sum_{R=1}^\infty\P\left(\sup_{x\in [-R,\,R]^d}u_t(x)>2U\right)<\infty.
\end{align*}
We can now use Borel-Cantelli and the fact that $B(0,\,R)\subset [-R,\,R]^d$ to finish the proof.
\end{proof}
\subsection{Proof of Theorem \ref{trichotomy}.}
\begin{proof}[Proof of Theorem \ref{trichotomy}.]
We split the proof into two parts. In the first part we assume that $\Lambda>0$, although $\Lambda=\infty$ is also possible as a particular case. Consider a sequence $\{x_n\}_{n \geq 1}\subset \R^d$ such that $\vert x_n\vert=n^{1/2}$ and all $x_n$ lie on a straight line through the origin. We next choose 
\begin{equation*}
\lambda\in (0,\,\Lambda),
\end{equation*}
and consider $$\quad t(j,\,n):=\frac{jT}{n}, \quad \text{ for } \quad j\in [\frac{n\tau}{T},\,n]\cap \Z.$$
We look at the following parameters $\tau$ and $T$ such that
\begin{equation*}
0<\tau<T:=\frac{K_2^{2/(2-\beta)} \lambda^{(4-\beta)/(2-\beta)}}{8^{2/(2-\beta)} },
\end{equation*}
where $K_2$ is the constant in the statement of Theorem \ref{tail:asymp}.
Then, by Theorem \ref{tail:asymp}, for all $\theta>0$, $t\in (\tau,\,T)$ and large enough $n$,
\begin{align*}
\P\left(\max_{j\in [\frac{n\tau}{T},\,n]}u_{t(j,\,n)}(x_n)>\theta\right)&\leq \sum_{j\in [\frac{n\tau}{T},\,n] \cap \Z}\P\left(u_{t(j,\,n)}(x_n)>\theta\right)\\
&\leq c \sum_{j\in [\frac{n\tau}{T},\,n] \cap \Z} \exp \left( -\frac{K_2 \lambda^{(4-\beta)/2} }{t(j,n)^{(2-\beta)/2}}\log \vert x_n \vert\right)\\
&\leq c n \exp \left( -\frac{K_2 \lambda^{(4-\beta)/2} }{2 T^{(2-\beta)/2}}\log n\right)
\\
&\leq \frac{c}{n^3}.
\end{align*}
An application of Borel-Cantelli lemma gives us
\begin{align*}
\lim_{n\rightarrow \infty} \max_{j\in [\frac{n\tau}{T},\,n]\cap \Z}u_{t(j,\,n)}(x_n)=0\quad \text{a.s}.
\end{align*}
We now use Proposition \ref{cty} to obtain for all $\theta>0$,
\begin{align*}
\P\Big\{\sup_{t\in (\tau,\,T)}&\min_{j\in [\frac{n\tau}{T},\,n]\cap \Z} |u_{t(j,\,n)}(x_n)-u_t(x_n)|>\theta\Big\}\\
&\leq \P\Big\{\sup_{t\in (\tau,\,T): |s-t|\leq T/n}|u_s(x_n)-u_t(x_n)|>\theta\Big\}\\
&\leq \frac{c_{T,k}}{n^{\tilde{\eta} k}}.
\end{align*}
By choosing $k$ large enough, we can apply Borel-Cantelli and use the above to see that 
\begin{align} \label{aux5}
\lim_{n\rightarrow \infty}\sup_{t\in (\tau,\,T)}u_t(x_n)=0\quad \text{a.s}.
\end{align}
We next use Proposition \ref{cty}, to get for all $\theta>0$,
\begin{align*}
\P\Big\{\sup_{t\in (\tau,\,T)}\sup_{x\in [x_n,\,x_{n+1}]} &|u_t(x_n)-u_t(x)|\geq \theta\Big\}\\
&\leq c n^{1/2} \P\Big\{\sup_{t\in (\tau,\,T)}\sup_{\vert x-y\vert \leq \frac{1}{n}} |u_t(x)-u_t(y)|\geq \theta\Big\}\\
&\leq c\frac{n^{1/2}}{n^{\eta k}}.
\end{align*}
We then take $k$ large enough, use Borel-Cantelli again and (\ref{aux5}) to conclude that
\begin{align*}
\lim_{\vert x\vert \rightarrow \infty}\sup_{t\in (\tau,\,T)}u_t(x)=0\quad \text{a.s}, 
\end{align*}
where in the above, $x$ tends to infinity along a fixed straight line. Since the line is arbitrary and $u$ is almost surely jointly continuous (Proposition \ref{cty}), it follows that
$$
\P\left( \sup_{x \in \R^d} u_t(x)<\infty, \text{ for all } t \in (\tau, T)\right)=1.
$$

We next assume that $\Lambda<\infty$. Fix $\theta>0$ and set
\begin{align*}
E_t(x):=\{\omega\in \Omega: u_t(x)\leq \theta \}\quad \text{for every}\quad t>0,  x \in \R^d.
\end{align*}
We will show that solution is almost surely unbounded for large enough times. 
Let
$$
\tau> (2K_1\Lambda^{2-\beta/2} )^{2/(2-\beta)}\quad \text{and} \quad T>\tau.
$$
According to Theorem \ref{tail:asymp}, for every $\lambda \in (\Lambda, (\tau^{1-\beta/2}/(2K_1) )^{2/(4-\beta)}]$, we can find a real number $n(\lambda, \theta)>1$ such that
\begin{align} \label{k1}
\P(E_t(x))\leq \left( 1-|x|^{-K_1 \lambda^{2-\beta/2}/t^{1-\beta/2}}\right)\leq  \left( 1-\frac{1}{|x|^{1/2}}\right),
\end{align}
 uniformly for all $\vert x \vert \geq n(\lambda, \theta)$ and $t\in (\tau,\,T)$.
 Consider the events
\begin{align*}
E^{(n)}_t(x):=\{\omega\in \Omega: U^{(n,\,n)}_t(x)\leq 2\theta \}\quad \text{for every}\quad x \in \R^d, n \geq 1.
\end{align*}
By Lemma \ref{approx}, we get
\begin{equation} \label{k2} \begin{split}
\sup_{t\in(\tau,\,T)}\P(E_t(x)\backslash E_t^{n}(x))&\leq \sup_{t\in(\tau,\,T)}\P\left(\left|u_t(x)-U^{(n,\,n)}_t(x) \right|\geq \theta\right)\\
&\leq \frac{c_{T,k}}{n^{\gamma k/2}}.
\end{split}
\end{equation}
Therefore,
\begin{align*}
\P\left(\bigcap_{x \in [n^{4},\,2n^{4}]^d} E_t(x)\right)&\leq \P\left(\bigcap_{\ell\in [n^{4},\,2n^{4}]^d\cap \Z^d}E_t(\ell)\right)\\
&\leq\P\left(\bigcap_{\ell\in [n^{4},\,2n^{4}]^d\cap \Z^d}E^{(n)}_t(\ell)\right)+\frac{c}{n^{\gamma k/2}},
\end{align*}
uniformly for all $n \geq 1$ and $t \in (\tau, T)$.
We will now look at the first term of the above display. Set $x_1:=(n^4,\ldots, n^4)\in \R^d$ and define iteratively for $j\geq 1$, 
\begin{align*}
x_{j+1}:=x_j+(2n^{3/2}t^{1/2},\ldots,2n^{3/2}t^{1/2}).
\end{align*}
Let
\begin{align*}
\gamma_n:=\max\left\{j\geq 1: x_{j,i}\leq 2n^{4}, \text{ for all } i=1,\ldots, d\right\},
\end{align*}
where $x_j=(x_{j,1}, \ldots, x_{j,d})$.
Observe that
\begin{align*}
\gamma_n\geq \frac{n^{5/2}}{2T^{1/2}}.
\end{align*}
By independence (Lemma \ref{ind}), (\ref{k2}) and (\ref{k1}), we get
\begin{align*}
\P\left(\bigcap_{\ell\in [n^{4},\,2n^{4}]^d\cap \Z^d}E^{(n)}_t(\ell)\right)&\leq \P\left(\bigcap_{j=1}^{\gamma_n}E^{(n)}_t(x_j)\right)=\prod_{j=1}^{\gamma_n}\P\left(E^{(n)}_t(x_j)\right)\\
&\leq \prod_{j=1}^{\gamma_n}[\P\left(E_t(x_j)\right)+\frac{c}{n^{\gamma k/2}}]\\
&\leq  \left[1-\frac{1}{\sqrt{2n^{4}}}+\frac{c}{n^{\gamma k/2}}\right]^{\gamma_n}.
\end{align*}
We now take $k$ larger if necessary to obtain
\begin{align*}
\P\left(\bigcap_{\ell\in [n^{4},\,2n^{4}]^d\cap \Z^d}E^{(n)}_t(\ell)\right)&\leq \exp(-c_1n^{1/2}).
\end{align*}
Combining the above estimates, we have for large enough $k$,
\begin{align} \label{k3}
\sup_{t\in (\tau,\,T)}\P\left(\sup_{x \in [n^{4}, 2n^{4}]^d}u_t(x)\leq \theta\right)\leq \frac{c}{n^{\gamma k/2}}.
\end{align}
We next write
\begin{align*}
\P&\left(\inf_{t\in (\tau,\,T)}\sup_{x\in[n^{4},\,2n^{4}]^d}u_t(x)\leq \theta \right)\leq \P\left(\inf_{1\leq i\leq n}\sup_{x\in[n^{4},\,2n^{4}]^d}u_{t_i}(x)\leq 2\theta \right)\\
&\qquad \qquad \qquad \qquad +\P\left(\inf_{|t-s|<1/n}\sup_{x\in[n^{4},\,2n^{4}]^d}\left|u_s(x)-u_t(x)\right|\geq \theta \right):=I_1+I_2.
\end{align*}
From (\ref{k3}), we can bound the first term as follows
\begin{align*}
I_1\leq \frac{c n}{n^{\gamma k/2}}.
\end{align*}
We now look at the second term. Using Proposition \ref{cty} we obtain that
\begin{align*}
I_2 \leq \sum_{k=1}^{1+n^4}\P\left(\inf_{|t-s|<1/n}\sup_{x\in (k,k+1)^d}\left|u_s(x)-u_t(x)\right|\geq \theta \right)\leq \frac{c}{n^\kappa},
\end{align*}
where $\kappa$ can be made as large as possible. Combining the above estimates, we conclude that 
\begin{align*}
\P\left(\inf_{t\in (\tau,\,T)}\sup_{x\in \R^d} u_t(x)<\theta \right)=0.
\end{align*}
For each $N\geq 1$, set
\begin{align*}
T_N:=\inf\{t>0: \sup_{x\in \R^d} u_t(x)\geq N \}. 
\end{align*}
And let $T:=\lim_{N\rightarrow \infty} T_N$. By the above computations, we have $t_1<T<t_2$, where $t_1$ and $t_2$ are deterministic constant depending on $\Lambda$. For any $t<T$, we have $\sup_{x\in \R^d} u_t(x)<\infty$ otherwise this would contradict the definition of $T$. On the other hand, if we have $t>T$, we then have $\sup_{x\in \R^d} u_t(x)=\infty$.
\end{proof}

\section{Proof of the comparison principle and strict positivity}

In order to prove the strong comparison principle (Theorem \ref{strongcomparison}), we need the next two preliminary results which are extensions of 
\cite[Lemmas 7.1 and 7.2]{ChenHuang} (see also \cite[Lemmas 4.1 and 4.3]{CK}). In particular, the proof of the next proposition is new compared with that of 
\cite[Lemma 7.1]{ChenHuang} or \cite[Lemma 4.1]{CK}.
\begin{proposition}\label{initial}
Let $M>0$. For all $R>0$ and $t>0$, there exist constants
$0<c_R<1$ and $1<m_0(t,R)<\infty$ such that  for all $m \geq m_0$,
\begin{align*}
\int_{B(0,\,R)}p_{s}(x-y)\,\d y\geq c_R \quad \text{for all }\quad (s,x) \in A_{m,t,R},
\end{align*} 
where
$$
A_{m,t,R}:=\{(s,x): x\in B(0,\,R+M(t/m)^{1/\alpha}) \quad\text{and} \quad  \frac{t}{2m}\leq s\leq \frac{t}{m}\}.
$$
\end{proposition}

\begin{proof}
We take $m$ large enough so that $(\frac{t}{m})^{1/\alpha} \leq R$. Then, using the the lower bound (\ref{ee2}), we obtain
\begin{align*}
\int_{B(0,\,R)}p_s(x-y)\,\d y
&\geq  \int_{B(0,\,R)\cap B(x,\,2M(t/m)^{1/\alpha})}p_s(x-y)\,\d y\\
&\geq c \left( \frac{t}{m}\right)^{d/\alpha} s^{-d/\alpha} \\
&\geq c 
\end{align*}
where the constant $c$ might depend on $R$ and $M$ but can be chosen to be strictly less than $1$.
\end{proof}
\begin{remark}\label{counterexample}
While the above result holds for all $\alpha\in (0,\,2)$, we were unable to use it to prove strict positivity for $\alpha<1$. Ideally, we would need Lemma 4.1 of \cite{CK} for $\alpha<1$ as well. Let $d=1$ and set $u_0(x)=1_{[0,\,1]}(x)$. We now choose $x=1+\frac{1}{m}$ and $\frac{t}{2m}\leq s\leq \frac{t}{m}$ so that we are exactly in the setting of Lemma 4.1 of \cite{CK}. We now use the usual bound on the heat kernel to find that 
\begin{align*}
\int_{\R}p_s(x-y)u_0(y)\,\d y&\leq \int_0^1\frac{s}{|x-y|^{1+\alpha}}\,\d y\\
&\leq c\frac{t}{m}\left(\frac{1}{|x|-1} \right)^\alpha\\
&\leq cm^{\alpha-1}.
\end{align*}
Since $\alpha<1$, the right hand of the above inequality goes to zero as $m\rightarrow \infty$. We have thus proved that for Lemma 4.1 of \cite{CK} cannot hold for $\alpha<1$ and therefore the strict positivity result cannot follow directly from the method of that paper. The  failure of this method can also be seen by the fact that as $m$ tends to infinity, $B^m_k$ defined in the proof of Theorem \ref{strongcomparison} tends to $B(0,\,R)$ when $\alpha<1$.
\end{remark}

\begin{proposition}\label{stoch}
Fix $R>0$, $t>0$ and $M>0$ and assume that
$$
u_0(x) \geq 1_{B(0,\,R)}(x).
$$
Then there exist positive constants  $c_1(R)$, $c_2(R)$, and $m_0(t,R)$ such that for all $m \geq m_0$,
\begin{align*}
\P(u_{s}(x)\geq  c_1 1_{B(0,\,R+M(t/m)^{1/\alpha})}(x)&\quad\text{for all}\quad \frac{t}{2m} \leq s\leq\frac{t}{m}\quad\text{and} \quad x\in \R^d )\\
&\geq 1-c_m,
\end{align*}
where
$$
c_m:=\exp\left(-c_2m^{(\alpha-\beta)/\alpha}[\log m]^{(2\beta-\alpha)/\alpha}\right).
$$ 
\end{proposition}
\begin{proof}From the mild formulation of the solution of the equation, we have 
\begin{align*}
u_{s}(x)=\int_{\R^d}p_{s}(x-y)u_0(y)\,\d y+\int_0^{s}\int_{\R^d}p_{s-l}(x-y)\sigma(u_l(y)) F(\d y\,\d l).
\end{align*}
By Proposition \ref{initial}, there exists a $0<c_1<1$ such that for large enough $m$,
\begin{align*}
\int_{\R^d}p_{s}(x-y)u_0(y)\,\d y\geq 2c_1 1_{B(0,\,R+M(t/m)^{1/\alpha})}(x)\quad \text{for all }\quad x\in \R^d \quad\text{and} \quad \frac{t}{2m} \leq s\leq \frac{t}{m}.
\end{align*}
By using the mild formulation and the above, we obtain
\begin{align*}
&\P(u_{s}(x)\leq c_1 1_{B(0,\,R+M(t/m)^{1/\alpha})}(x) \quad\text{for some}\quad \frac{t}{2m}\leq s\leq\frac{t}{m})\\
&\leq \P\Big(\int_0^{s}\int_{\R^d}p_{s-l}(x-y)\sigma(u_l(y)) F(\d y\,\d l)<- c_1  \quad\text{for some}\quad (s,x) \in A_{m,t,R}\Big)\\
&\leq \P\Big(\left|\int_0^{s}\int_{\R^d}p_{s-l}(x-y)\sigma(u_l(y)) F(\d y\,\d l)\right|>c_1\quad\text{for some}\quad  (s,x) \in A_{m,t,R}\Big).\end{align*}
The term in the above display can now be bounded by 
\begin{align*}
c_1^{-k}\E\sup_{(s,x) \in A_{m,t,R} }\left|\int_0^{s}\int_{\R^d}p_{s-l}(x-y)\sigma(u_l(y)) F(\d y\,\d l) \right|^k.
\end{align*}
The above in turn can be bounded using Proposition \ref{cty} to obtain 
\begin{align*}
&\E\sup_{(s,x) \in A_{m,t,R}}\left|\int_0^{s}\int_{\R^d}p_{s-l}(x-y)\sigma(u_l(y)) F(\d y\,\d l) \right|^k\\
&\leq c \rho^{\tilde{\eta}k}\exp(Ak^{(2\alpha-\beta)/(\alpha-\beta)}\rho),
\end{align*}
where $\rho:=t/m$ and $\tilde{\eta}:=\frac{\alpha-\beta}{2\alpha}$. We now optimise the above quantity with respect to $k$ and combine all our estimates to end up with the result. See \cite{ChenHuang} and \cite{CK} for details.

\end{proof}

\subsection{Proof of Theorem \ref{strongcomparison}}
We next prove the strong comparison principle. We leave it to the reader ro consult \cite{Mueller} for the original idea and to \cite{ChenHuang} and \cite{CK} for further details.

\begin{proof} 
It suffices to show that if $u_0$ has compact support then $u_t(x)>0$ for all $t>0$ and $x \in \R^d$ a.s. The general case will follow as in \cite{ChenHuang} and \cite{CK}. 
Assume that $u_0(x)=1_{B(0,\, R)}(x)$, for some $R>0$. Choose $M>0$, $t>0$ and $m>0$.  Define for $k=1,\ldots,2m-1$,
\begin{align*}
A_k:=\left\{u_{s}(x)\geq c_1^{k+1}1_{B^m_k}(x)\quad\text{for all} \quad s\in \Big[\frac{kt}{2m},\,\frac{(k+1)t}{2m}\Big] \quad\text{and}\quad x\in \R^d\right\},\end{align*}
where $B^m_k=B(0, \, R+kM(t/m)^{1/\alpha})$ and $c_1$ is as in Proposition \ref{stoch}. 
It is clear that if $\alpha>1$, then as $m$ gets large, the sets $B^m_k$ cover the whole space. For $\alpha=1$, the sets $B^m_k$ cover $B(0, \, R+ Mt^{1/\alpha})$.  
We write
\begin{equation} \begin{split}\label{h1}
&\P(u_{s}(x)>0\quad\text{for all}\quad t/2\leq  s\leq t\quad\text{and}\quad x\in B(0,\, M/2))\\
&\qquad \qquad \geq \lim_{m\rightarrow \infty}\P(\cap_{1\leq k\leq 2m-1}A_k)\\
&\qquad\qquad=\lim_{m\rightarrow \infty}\P(A_1) \prod_{2\leq k\leq 2m-1}\P(A_k|A_{k-1}\cap\cdots \cap A_1).
\end{split}
\end{equation}
Proposition \ref{stoch} can be used to obtain 
\begin{equation}\label{h2}
\P(A_1)\geq 1-c_m,
\end{equation}
whenever $m$ is large enough since $c_1>c_1^2$. 
On the other hand, on the event $A_{k-1}$, $k \geq 2$,
$$
u_{\frac{kt}{2m}}(x)\geq c_1^{k}1_{B^m_{k-1}}(x), \quad \text{ for all } \; x\in \R^d.
$$
By the Markov property, $\{u_{s+\frac{kt}{2m} }(x), s \geq 0, x \in \R^d\}$ solves (\ref{main-eq}) with the time-shifted noise $\dot{F}_{k}(s,\,x):=\dot{F}(s+\frac{kt}{2m},\,x)$ starting from $u_{\frac{kt}{2m}}(x)$.
Let $\{v^{(k)}_{s}(x), s \geq 0, x \in \R^d\}$ be the solution to (\ref{main-eq}) with the time-shifted noise $\dot{F}_{k}(s,\,x)$, $\sigma$ replaced by $\sigma_k(x)=c_1^{-k} \sigma(c_1^kx)$,
and initial condition $1_{B^m_{k-1}}(x)$.   On one hand, by Proposition \ref{stoch} we get that
\begin{align*}
\P\big(v_{s}^{(k)}(x)\geq c_1 1_{B^m_k}(x)\quad\text{for all}\quad s\in \big[\frac{t}{2m},\frac{t}{m}\big]\quad \text{and}\quad x\in \R^d\big)\geq 1-c_m,
\end{align*}
whenever $m$ is large enough.
On the other hand, by Markov property and the weak comparison principle (Theorem \ref{weakc}) we see that on $A_{k-1}$, $u_{s+kt/(2m)}(x)\geq c_1^kv_{s}^{(k)}(x)$ for all $x\in \R^d$ and $s\geq 0$. We therefore have 
\begin{align*}
\P(A_k|\mathcal{F}_{kt/(2m)})\geq 1-c_m\quad \text{on}\quad A_{k-1}.
\end{align*}
And hence
\begin{equation} \label{h3}
\P(A_k|A_{k-1}\cap\cdots \cap A_1)\geq 1-c_m.
\end{equation}
From (\ref{h1}), (\ref{h2}) and (\ref{h3}), we conclude that
\begin{align*}
\P(u_{s}(x)>0\quad\text{for all}\quad t/2\leq  s\leq t\quad\text{and}\quad x\in B(0,\, M/2))\geq (1-c_m)^{2m-1} \rightarrow 1,
\end{align*}
as $m \rightarrow \infty$.
Since the above holds  any arbitrary $t>0$ and $R,M>0$, the proof is complete.
\end{proof}

\subsection{Proof of Theorem \ref{positivity}}
\begin{proof}
The proof is very similar to those in \cite{ChenHuang}, \cite{CK} and \cite{CJKCorr}, using the strong Markov property (Lemma \ref{markov}) and the weak comparison principle (Theorem \ref{weakc}).  So we omit it.
\end{proof}
{\bf \begin{remark}
As mentioned in the introduction, the above comparison theorem and strict positivity results are shown under the assumption that the initial conditions are bounded functions. A wider class of initial conditions could be studied as in \textnormal{\cite{ChenHuang}} and \textnormal{\cite{CK}}.  We leave it for further work.
\end{remark}}
\bibliography{Foon}
\end{document}